\documentclass[reqno,12pt]{amsart}

\newenvironment{nouppercase}{%
  \renewcommand{\uppercasenonmath}[1]{}}{}

\usepackage{amsmath,amssymb,amsthm,amsfonts,mathrsfs}
\usepackage{fullpage}
\usepackage{colonequals}
\usepackage{enumerate} 
\usepackage[colorlinks=true,
linkcolor=blue,
anchorcolor=blue,
citecolor=red
]{hyperref}
\allowdisplaybreaks 
\newtheorem{theorem}{Theorem}[section]
\newtheorem{lemma}[theorem]{Lemma}

\theoremstyle{definition}

\theoremstyle{remark}
\newtheorem{remark}[theorem]{Remark}

\newcommand{\Q}{\mathbb{Q}}

\newcommand{\C}{\mathbb{C}}
\newcommand{\E}{\mathbb{E}}
\newcommand{\F}{\mathbb{F}}
\renewcommand{\P}{\mathbb{P}}

\newcommand{\cA}{\mathcal{A}}
\newcommand{\cB}{\mathcal{B}}
\newcommand{\cC}{\mathcal{C}}

\newcommand{\cO}{\mathcal{O}}
\newcommand{\fp}{\mathfrak{p}}

\newcommand{\on}{\operatorname}

\renewcommand{\mod}[1]{\,(\mathrm{mod}\,#1)}
\newcommand{\of}[1]{\left(#1\right)}
\newcommand{\set}[1]{\left\{#1\right\}}
\newcommand{\norm}[1]{\left\Vert#1\right\Vert}

\author{Huixi Li}
\address{School of Mathematical Sciences and LPMC, Nankai University, Tianjin 300071, China}
\email{lihuixi@nankai.edu.cn}

\author{Biao Wang}
\address{School of Mathematics and Statistics, Yunnan University, Kunming, Yunnan 650500, China}
\email{bwang@ynu.edu.cn}

\author{Shaoyun Yi}
\address{School of Mathematical Sciences, Xiamen University, Xiamen, Fujian 361005, China}
\email{yishaoyun926@xmu.edu.cn}

\date{\today}

\makeatletter
\@namedef{subjclassname@2020}{\textup{}2020 Mathematics Subject Classification}
\makeatother

\title{On the minimum modulus problem in number fields}

\begin{document}
	
\begin{abstract} The minimum modulus problem on covering systems was posed in 1950 by Erd\H{o}s, who asked whether the minimum modulus of a covering system with distinct moduli is bounded. In 2007,  Filaseta, Ford, Konyagin, Pomerance and Yu affirmed it if the reciprocal sum of the moduli of a covering system is bounded. Later in 2015, Hough resolved this problem by showing that the minimum modulus in any covering system with distinct moduli is at most $10^{16}$. In 2022, Balister, Bollob\'as, Morris, Sahasrabudhe and Tiba reduced this bound to $616,000$ by developing a versatile method called the distortion method. Recently, Klein, Koukoulopoulos and Lemieux generalized Hough's result by using a suitable modification of the distortion method. In this paper, we develop the distortion method further by introducing the theory of probability measures associated to an inverse system. Following Klein et al.'s work, we provide a solution to Erd\H{o}s' minimum modulus problem in number fields. As an application, we prove that the $j$-th smallest norm in a minimal covering system of a number field with distinct moduli is bounded.
 
\end{abstract}
	
\begin{nouppercase}
\maketitle
\end{nouppercase}

\section{Introduction and statement of results}

A \textit{covering system} is
a finite collection of arithmetic progressions that covers the integers.
It was first introduced by Erd\H{o}s \cite{Erdos1950}
in 1950 to answer a question of Romanoff.  And in his collections of open problems \cite{Erdos1957,  Erdos1963, Erdos1973, EG1980}, Erd\H{o}s posed a number of problems concerning covering systems. One of them is the well-known minimum modulus problem: is there a uniform upper bound on the smallest modulus of all covering systems with distinct moduli? In 2007,  Filaseta, Ford, Konyagin, Pomerance and Yu \cite{FFSPY2007} made a significant process towards this problem, showing that the reciprocal sum of the moduli  grows with the minimum modulus. Building on their work, Hough \cite{Hough2015} made a remarkable breakthrough and gave an affirmative answer to Erd\H{o}s' minimum modulus problem in 2015, showing that the minimum modulus of any covering system with distinct moduli is at most $10^{16}$. In 2022, Balister, Bollob\'as, Morris, Sahasrabudhe and Tiba \cite{BBMRS2022} reduced this bound to $616,000$ by developing a general method for bounding the density of the uncovered set. This method is named by them \textit{the distortion method}; see \cite{BBMST2020_2} for more details about it. Recently, Cummings, Filaseta, and Trifonov \cite{CFT2022} reduced the bound to 118 for  covering systems with distinct squarefree moduli. They also proved in general the $j$-th smallest modulus in a minimal covering system with distinct moduli is bounded by some unspecified constant $C_j, j\ge1$. This result was improved by Klein, Koukoulopoulos and Lemieux \cite{KKL2022} later in the sense that they gave a specified bound for the $j$-th smallest modulus. In fact, they achieved the following generalization of Hough's theorem by using a suitable modification of the distortion method.  For a covering system $\cA$, the \textit{multiplicity} of $\cA$ is defined to be the largest number of times that a modulus appears in $\cA$.  

\begin{theorem}[\text{\cite[Theorem~3]{KKL2022}}]\label{KKL Thm3}
Let $\cA$ be a covering system of multiplicity $s$. Then there exists an absolute constant $c > 0$ such that the smallest modulus of $\cA$ is at most $\exp \big(c \log^2(s+1)/\log\log(s+2)\big)$. 
\end{theorem}

We refer the reader to Balister's recent survey \cite{Balister2024} on these beautiful results.

In this paper, our goal is to establish the generalization of Theorem~\ref{KKL Thm3} to the covering systems of the ring of integers of an arbitrary number field. In fact, the covering systems for number fields have already been actively studied in the literature; see for example \cite{Kim2009, Kim2012, JD2014, BB2020}. To explain our main result, we first introduce some notation in the number fields setting. 

Let $K/\Q$ be a number field. Let $\cO_K$ be the ring of integers in $K$. For any ideal $I\subset \cO_K$, we denote by $\norm{I}$ the norm of $I$, i.e., $\norm{I}=[\cO_K\colon I]$. Let $\cA=\set{a_1+I_1, \dots, a_n+I_n}$ be a finite collection of congruence classes with $a_i\in\cO_K$, $1\le i \le n$. Let 
\[
m(\cA)= \max_{I\subset \cO_K} \#\set{1 \le j \le n: I_j=I}
\]
be the multiplicity of $\cA$. We say $\cA$ is a \textit{covering system} of $\cO_K$ if $\cO_K=\bigcup_{i=1}^n(a_i+I_i)$; see \cite[Page~124]{Kim2009} for an example.  Now we are ready to state our main theorem.

\begin{theorem}\label{mainthm}
For a number field $K/\Q$, let $\cA=\set{a_1+I_1, \dots, a_n+I_n}$ be a covering system of multiplicity $s$ with ideals $I_1,\dots,I_n$ in $\cO_K$. Then there exists a constant $c_K>0$ depending only on $K$ such that
\[
\min_{1\le i \le n}\norm{I_i} \le \exp\big(c_K\log^2(s+1)/\log\log(s+2)\big).
\]
\end{theorem}
\begin{remark} If one can generalize Crittenden and Vanden-Eynden's theorem \cite{CV1970} to any number field, then one may obtain a generalization of Theorem~1 in Klein, Koukoulopoulos and Lemieux's work \cite{KKL2022} to the setting of number fields.  See \cite{BBMST2020_1} for a short proof of Crittenden and Vanden-Eynden's theorem.
\end{remark}

\begin{remark}
	Let $\F_q$ be a finite field. Let $\F_q[x]$ be the polynomial ring over $\F_q$. In 2024, the authors of this paper and co-author Wang \cite{LWWY2024} proved an analogue of Theorem~\ref{mainthm} for the covering systems of $\F_q[x]$. Moreover, in \cite{LWWY2025} they showed that there is no covering system of $\F_q[x]$ with distinct moduli for $q\ge 759$. Recently, this lower bound is reduced to $q>73$ in \cite{Wang2026ijnt}.
\end{remark}

When we take $s = 1$ in Theorem~\ref{mainthm}, then it provides a solution to Erd\H{o}s' minimum modulus problem in number fields. That is,  in a given number field, the smallest norm of the moduli of all covering systems with distinct moduli is uniformly bounded. To prove Theorem~\ref{mainthm}, we follow the ideas in the proof of Theorem~3 in  \cite{KKL2022} and derive analogues of the necessary lemmas and theorems in \cite{BBMRS2022, KKL2022} in the number fields setting. First, in Section~\ref{sec_distortion_method} we present a notable generalization of the distortion method  by introducing the theory of probability measures associated to an inverse system; see Theorem~\ref{thm_distortion_method}. Then we restrict it to the number fields setting for our purpose. One may apply Theorem~\ref{thm_distortion_method} to prove the analogue of Theorem~\ref{KKL Thm3} in other situations; say function fields.  In Sections~\ref{boundind_moments}-\ref{sec_pf_mainthm},  we investigate the first and second moments appearing in Theorem~\ref{thm_distortion_method_nf} similarly as in \cite{KKL2022} and then complete the proof of our main theorem. Finally, in Section~\ref{sectjthsmallestnorm}, we apply our main theorem to study the $j$-th smallest norm in a minimal covering system of a number field with distinct moduli, where we prove Theorem~\ref{thmjthsmallestnorm}. 

\section{The distortion method}
\label{sec_distortion_method}

\subsection{Probability measures associated to an inverse system}
\label{sec_distortion_method_general}

In this section, we construct a sequence of probability measures on a finite set with an inverse system. Let $S$ be a finite set, $T$ be an arbitrary set, and $\pi:S\to T$ be a map. We say a function $f:S\to \C$ is \textit{$\pi$-measurable} if $f(x) = f(x')$ whenever $\pi(x)=\pi(x')$ for any $x,x'\in S$. For any subset $B\subseteq S$, we say $B$ is \textit{$\pi$-measurable}, if the following property holds: for any $x\in B, x'\in S$ if $\pi(x')=\pi(x)$, then $x'\in B$. Let $\P$ be a $\pi$-measurable probability measure on $S$. Let $J\ge1$ be an integer. Let $B_1, \dots, B_J$ be subsets of $S$, $T_1,\dots, T_J$ be arbitrary sets, and $\pi_j: S \to T_j$ be arbitrary maps for $1\le j \le J$. We say $\set{\pi_j}$ is a \textit{$\pi$-inverse system} (or simply an \textit{inverse system}) if the following property holds: for any $1\le j\le J$ and $x,x'\in S$, if $\pi_j(x)=\pi_j(x')$, then $\pi_{j-1}(x)=\pi_{j-1}(x')$, where $\pi_0=\pi$. Let $\delta_1,\dots, \delta_J\in [0,1/2]$ be parameters. With these definitions in hand, following the steps in \cite{BBMRS2022,KKL2022} we can construct a sequence of $\pi_j$-measurable\footnote{In \cite{BBMRS2022,KKL2022}, it is called  $Q_j$-measurable.} probability measures $\P_j (1\le j \le J)$ on $S$ with the triples $\set{(\pi_j, B_j, \delta_j): 1\le j\le J}$ and the initial probability measure $\P$ as follows. 

First, we take $\P_0=\P$. Next, for $1\le j \le J$, suppose we have already constructed a $\pi_{j-1}$-measurable probability measure $\P_{j-1}$ on $S$. For all $x\in S$ we define
\[
  \alpha_j(x) \colonequals \frac{|F_{j-1}(x)\cap B_j|}{|F_{j-1}(x)|},
\]
where $F_{j-1}(x)=\{x'\in S:\pi_{j-1}(x')=\pi_{j-1}(x)\}$. 
Note that $\alpha_j$ is $\pi_{j-1}$-measurable, since
$F_{j-1}(x')=F_{j-1}(x)$ holds for any $x', x\in S$ satisfying $\pi_{j-1}(x')=\pi_{j-1}(x)$. 
If $\delta_j=0$, we simply set $\mathbb P_j=\mathbb P_{j-1}$.
If $0<\delta_j\le 1/2$, we define $\mathbb P_j$ by
\[
\mathbb P_j(x)\colonequals \mathbb P_{j-1}(x)\cdot
\begin{cases}
\dfrac{1_{x\notin B_j}}{1-\alpha_j(x)},&
\text{if }\alpha_j(x)<\delta_j,\\[6pt]
\dfrac{\alpha_j(x)-1_{x\in B_j}\delta_j}
{\alpha_j(x)(1-\delta_j)},&
\text{if }\alpha_j(x)\ge \delta_j.
\end{cases}
\]
One may easily check that $\P_j$ is $\pi_j$-measurable, and we have
\begin{equation}\label{eqn_measurable}
    \P_j(F_{j-1}(x)) = \P_{j-1}(F_{j-1}(x)) 
\end{equation}
for all $x\in S$. It follows that $\P_j(S)=\P_{j-1}(S)=1$, and hence it is a probability measure on $S$. Thus, by induction we obtain a sequence of probability measures $\set{\P_j: 1\le j \le J}$. 

Now, for $1 \leq j \leq J$, we define the expectations
\[
\E_j[f]\colonequals \sum_{x\in S} f(x)\P_j(x)
\]
for all functions $f\colon S\to \C$, and we set
\[
M_j^{(1)}\colonequals\E_{j-1}[\alpha_j] 
\quad\text{and}\quad
M_j^{(2)}\colonequals\E_{j-1}[\alpha_j^2]
\]
as the first and second moments. Then the following theorem is the key result of the distortion method in \cite[Theorem~3.1]{BBMRS2022} in this general setting. 

\begin{theorem}\label{thm_distortion_method}
  Assume the above notation.  Let $B_j$ be $\pi_j$-measurable subsets of $S$ for all $1\le j \le J$. If 
	\[
	\sum_{j=1}^J
	\min\left\{  M_j^{(1)},\frac{M_j^{(2)}}{4\delta_j(1-\delta_j)}\right\}<1 ,
	\]
	then 
 \[
 \P_J\of{\bigcup_{1\le j \le J} B_j}<1.
 \]
\end{theorem}

\begin{proof} Following the proof of \cite[Theorem~3.1]{BBMRS2022}, we first prove that
\begin{equation} \label{eqn_pjbj}
    \P_j(B_j)\le \min\set{M_j^{(1)},\frac{M_j^{(2)}}{4\delta_j(1-\delta_j)}}
\end{equation}
for all $1\le j \le J$. Indeed, by the construction of $\P_j$, we have $\P_j(B_j) \le \P_{j-1}(B_j)=\E_{j-1}[\alpha_j]$.  Moreover, by the inequality $\max\set{a-d,0}\le a^2/4d$ for all $a,d>0$ we have 
\begin{align*}
    \P_j(B_j) & = \sum_{x\in B_j} \max \set{0,\frac{\alpha_j(x)-\delta_j}{\alpha_j(x)(1-\delta_j)}}\P_{j-1}(x) \\
    & = \frac{1}{1-\delta_j}\sum_{x\in S} \max \set{0,\alpha_j(x)-\delta_j}\P_{j-1}(x) \\
    & \le \frac{1}{1-\delta_j}\sum_{x\in S} \frac{\alpha_j(x)^2}{4\delta_j}\P_{j-1}(x)\\
    & = \frac{\E_{j-1}[\alpha_j^2]}{4\delta_j(1-\delta_j)}.
\end{align*}
Hence \eqref{eqn_pjbj} holds.

Now, we assume $B_j$ is $\pi_j$-measurable for all $1\le j \le J$. Then we have $F_j(x)\subset B_j$ for all $x\in B_j$. So each $B_j$ is a disjoint union of some subsets of the form $F_j(x)$.  By \eqref{eqn_measurable}, we see that $\P_{j+1}(B_j)=\P_j(B_j)$. Since $\set{\pi_j}$ is an inverse system, $B_j$ is $\pi_i$-measurable as well for any $i>j$. It follows by induction that we have $\P_j(B_j)=\P_{j+1}(B_j)=\cdots=\P_J(B_j)$. Thus, summing both sides of \eqref{eqn_pjbj} over $1\le j \le J$ we obtain the desired result.
\end{proof}

\subsection{The distortion method for number fields}

Let $K/\Q$ be a number field.  Let $\cA=\set{a_1+I_1, \dots, a_n+I_n}$ be a collection of congruence classes of multiplicity $s$ for $a_1,\dots, a_n\in \cO_K$ and ideals $I_1,\dots,I_n$ in $\cO_K$ with $\norm{I_1}\le \cdots\le \norm{I_n}$. Let $Q=I_1\cap\cdots\cap I_n$. Then $\cO_K/Q$ is a finite set. Suppose we have the prime ideal decomposition $Q=\prod_{i=1}^J\fp_i^{\nu_i}$ with $\norm{\fp_1} \le \cdots \le \norm{\fp_J}$. For $1\le j \le J$, we take $Q_j=\prod_{i=1}^j\fp_i^{\nu_i}$ and  let $\pi_j:\cO_K/Q\to \cO_K/Q_j$ be the natural projections. Then $\set{\pi_j}$ is an inverse system. For $1\le j \le J$, let 
\[
\cB_j\colonequals\bigcup_{\substack{1\le i\le n\\ I_i|Q_j, I_i\nmid Q_{j-1}}} \set{a+Q: a \equiv a_i \mod{I_i}}.
\]
Clearly, $\cB_j$ is a $\pi_j$-measurable subset of $\cO_K/Q$ for each $1\le j\le J$. 

Now, take $\P_0$ to be the uniform probability measure on $\cO_K/Q$. Then we can construct a sequence of probability measures $\P_1, \cdots, \P_J$ on $\cO_K/Q$ with respect to the triples $\set{(\pi_j, \cB_j, \delta_j): 1\le j\le J}$ and $\P_0$ by taking $S=\cO_K/Q$ and $B_j=\cB_j$ in Section~\ref{sec_distortion_method_general}. Here  $\delta_1,\dots, \delta_J\in [0,1/2]$ are parameters.  By Theorem~\ref{thm_distortion_method}, we have the following result of the distortion method in \cite{BBMRS2022, KKL2022} for number fields.

\begin{theorem}\label{thm_distortion_method_nf}
Let $K/\Q$ be a number field.  Let $\cA$ be a finite collection of congruence classes in $\cO_K$. Assume the above construction. If
	\[
	\sum_{j=1}^J
	\min\left\{M_j^{(1)},\frac{M_j^{(2)}}{4\delta_j(1-\delta_j)}\right\}<1 ,
	\]
	then $\cA$ does not cover the algebraic integers of $K$. 
\end{theorem}

To prove Theorem~\ref{mainthm} via Theorem~\ref{thm_distortion_method_nf}, we take $S=\cO_K/Q$ and $B_j=\cB_j$ in Section~\ref{sec_distortion_method_general} and assume the above notation in the following sections.

\section{Bounding the first and second moments}\label{boundind_moments}

We will estimate $M_j^{(1)}$ and $M_j^{(2)}$ in this section. Recall that 
\[
\alpha_j(x)= \frac{|F_{j-1}(x)\cap \cB_j|}{|F_{j-1}(x)|},
\]
where $F_{j-1}(x)=\set{a+Q: a\equiv c\mod{Q_{j-1}}}$ for any $x=c+Q\in\cO_K/Q$ with $c\in\cO_K$.

\begin{lemma} \label{lem:B in fibers}
For $x\in\cO_K/Q$ and $1\le j \le J$, we have
\[
\alpha_j(x) \le \sum_{r=1}^{\nu_j} \sum_{H|Q_{j-1}}  \sum_{\substack{1\le i\le n \\ I_i=H \fp_j^r}} \frac{1_{x \subseteq a_i+H}}{\norm{\fp_j}^r}. 
\]
\end{lemma}

\begin{proof} First, we have $|F_{j-1}(x)|=\norm{Q}/\norm{Q_{j-1}}$. Write $x=c+Q$ for some $c\in\cO_K$. Then for any $a\in F_{j-1}(x)\cap\cB_j$, there exists some $1\le i \le n$ such that $a\equiv a_i\mod{I_i}$ and $I_i|Q_j, I_i\nmid Q_{j-1}$. Moreover, we have $a\equiv c \mod{Q_{j-1}}$. Then it follows that

\begin{equation*}
    \alpha_j(x)=\frac{|F_{j-1}(x)\cap\cB_j|}{\norm{Q}/\norm{Q_{j-1}}}\leq \frac{\norm{Q_{j-1}}}{\norm{Q}}\sum_{\substack{1\le i\le n\\ I_i|Q_j, I_i\nmid Q_{j-1}}}\sum_{\substack{a \mod{Q}\\ a\equiv c\mod{Q_{j-1}}\\ a\equiv a_i\mod{I_i}}}1.
\end{equation*} 

For each $i$ with $I_i|Q_j, I_i\nmid Q_{j-1}$, we write uniquely $I_i=H\fp_j^r$ with $H|Q_{j-1}$ and $1\leq r\leq \nu_j$. Since $a\equiv a_i\mod{I_i}$ and $a\equiv c\mod{Q_{j-1}}$, we have $c\equiv a_i\mod{H}$ and hence $x\subseteq a_i+H$. By $a\equiv a_i\mod{I_i}$ we have $a\equiv a_i\mod{\fp_j^r}$. It follows by the Chinese Remainder Theorem that $a$ lies in some congruence class mod $Q_{j-1}\fp_j^r$. In particular, there are $\norm{Q}/\norm{Q_{j-1}\fp_j^r}$ choices for $a \mod{Q}$ lying in a congruence class mod $Q_{j-1}\fp_j^r$. Then the desired estimate follows by these conditions. 
\end{proof}

\begin{lemma}\label{lemma applied in proof of Lemma 4.2}
For each $0\le j \le J$,  $a\in\cO_K$, and any ideal $I\subset \cO_K$ such that $I|Q$, we have
\[
\P_j(a+I)\le \frac{1}{\norm{I}}\prod_{\fp_i|I, i\le j}\frac{1}{1-\delta_i}.
\]
\end{lemma}

\begin{proof} We follow the proof of \cite[Lemma~3.4]{BBMRS2022} by induction on $j$. Since $\P_0$ is the uniform probability measure, we have $\P_0(a+I)=1/\norm{I}$. So the estimate holds for the case $j=0$. Let $j\ge1$, and assume the estimate holds for $\P_{j-1}$. By the definition of $\P_j$, we notice that 
$$\P_j(x+Q)\le\frac{1}{1-\delta_j}\P_{j-1}(x+Q), \forall x\in \cO_K.$$ 

Suppose that $\fp_j | I$, then we have 
\begin{align*}
    \P_j(x+I)&\le \frac{1}{1-\delta_j}\P_{j-1}(a+I)\\
    &\le \frac{1}{\norm{I}}\cdot \frac{1}{1-\delta_j}\prod_{\fp_i|I, i< j}\frac{1}{1-\delta_i}\\
    &=\frac{1}{\norm{I}}\prod_{\fp_i|I, i\le j}\frac{1}{1-\delta_i}.
\end{align*}

If $\fp_j \nmid I$, then $I=ML$, where $M=\gcd(I,Q_j)=\gcd(I,Q_{j-1})$ and $\gcd(L,Q_{j-1})=\cO_K$. It follows that
\begin{align*}
    \P_j(a+I)&=\frac{\P_j(a+M)}{\norm{L}}=\frac{\P_{j-1}(a+M)}{\norm{L}}\\
    &\le \frac{1}{\norm{L}\norm{M}} \prod_{\fp_i|M, i< j}\frac{1}{1-\delta_i}=\frac{1}{\norm{I}}\prod_{\fp_i|I, i\le j}\frac{1}{1-\delta_i},
\end{align*}
as required.   
\end{proof}

Now, we are ready to prove the following bounds of $M_j^{(1)}$ and $M_j^{(2)}$.

\begin{lemma}\label{lem:moments} 
	Assume the above notation, let $s=m(\cA)$, and let $j\in\{1,2,\dots,J\}$. 
	\begin{enumerate}
	\item If $\delta_i=0$ for $i\in\{1,\dots,j-1\}$, then 
	\[
	M_j^{(1)} \le  s \sum_{\substack{\norm{I}\ge \norm{I_1} \\ P(I)=\norm{\fp_j}}} \frac{1}{\norm{I}},
	\]
	where $P(I)$ is the largest norm of the prime ideals dividing $I$.
	\item We have
	\[
	M_j^{(2)} \ll \frac{s^2(\log \norm{\fp_j})^6}{\norm{\fp_j}^2}.
	\]
		\end{enumerate}
\end{lemma}

\begin{proof} 
Let $k$ be a positive integer. By Lemma~\ref{lem:B in fibers} we have
\begin{align}
    \E_{j-1}(\alpha_j^k)&=\sum_{x\in \cO_K/Q}\alpha_j^k(x)\P_{j-1}(x)\nonumber\\
    &\le \sum_{1\le r_1,\dots, r_k\le \nu_j} \sum_{H_1,\dots, H_k | Q_{j-1}} \sum_{\substack{1\le i_1,\dots, i_k\le n\\ I_{i_\ell}=H_\ell \fp_j^{r_\ell}, \forall \ell}} \frac{\P_{j-1}\left(\bigcap_{\ell=1}^k(a_{i_\ell}+H_\ell)\right)}{\norm{\fp_j}^{r_1+\dots+r_k}}.\label{kth moment}
\end{align}
Observe that $\norm{H_\ell\fp_j^{r_\ell}}\geq \norm{I_1}$ for all $1\le \ell\le k$, since $\norm{I_i}\geq \norm{I_1}$ for all $1\le i\le n$. Moreover, given $r_1, \ldots, r_k$ and $H_1, \ldots, H_k$, there are at most $s^k$ choices for $i_1, \ldots, i_k$ with $I_{i_\ell}=H_\ell \fp_j^{r_\ell}$. It follows by the Chinese Remainder Theorem that for each such choice of $i_1, \ldots, i_k$ the set $\bigcap_{\ell=1}^k(a_{i_\ell}+H_\ell)$ is either empty, or a congruence class mod $\bigcap_{\ell=1}^kH_\ell$.  Using Lemma~\ref{lemma applied in proof of Lemma 4.2} we can see that 
\begin{equation}\label{prob of intersection of APs}
    \P_{j-1}\left(\bigcap_{\ell=1}^k(a_{i_\ell}+H_\ell)\right) \le \frac{1}{\norm{\bigcap_{\ell=1}^k H_\ell}}\prod_{\substack{i\leq j-1\\ \fp_i|\bigcap_{\ell=1}^k H_\ell}}(1-\delta_i)^{-1}
\end{equation}
for at most $s^k$ choices of $1\le i_1,\dots, i_k\le n$.
Then plugging \eqref{prob of intersection of APs} into \eqref{kth moment} we obtain
\[
\E_{j-1}(\alpha_j^k)\le s^k \sum_{1\le r_1,\dots, r_k\le \nu_j} \sum_{\substack{H_1,\dots, H_k | Q_{j-1}\\ \norm{H_\ell\fp_j^{r_\ell}}\ge \norm{I_1}, \forall \ell}} \frac{1}{\norm{\bigcap_{\ell=1}^k H_\ell}\norm{\fp_j}^{r_1+\dots+r_k}}\prod_{\substack{i\leq j-1\\ \fp_i|\bigcap_{\ell=1}^k H_\ell}}(1-\delta_i)^{-1}.
\]
In particular, let $k=1$ and $\delta_i=0$ for all $i\leq j-1$, it follows that

\begin{align*}
    M_j^{(1)} & \le s\sum_{1\le r \le \nu_j} \sum_{\substack{H|Q_{j-1}\\ \norm{H\fp_j^r}\ge \norm{I_1}}} \frac{1}{\norm{H\fp_j^r}}  \\
    & \le s\sum_{\substack{\norm{I}\ge \norm{I_1}\\ P(I)=\norm{\fp_j}}}\frac{1}{\norm{I}}.
\end{align*}
This proves part (1) of the lemma. 

Next, let $k=2$ and $\delta_i\in[0,1/2]$ for all $i$. It is clear that
\begin{equation*}
    \prod_{\fp_i| H_1\cap H_2}(1-\delta_i)^{-1}\leq 2^{\omega(H_1\cap H_2)},
\end{equation*}
where $\omega(H_1\cap H_2)$ is the number of distinct prime ideals of $H_1\cap H_2$. Hence, we have
\begin{align*}
   M_j^{(2)} & \le s^2\sum_{1\leq r_1, r_2\leq \nu_j}\sum_{H_1, H_2\mid Q_{j-1}}\frac{2^{\omega(H_1\cap H_2)}}{\norm{H_1\cap H_2}\norm{\fp_j}^{r_1+r_2}}\\
   &\le \frac{s^2}{(\norm{\fp_j}-1)^2}\sum_{H_1, H_2\mid Q_{j-1}}\frac{2^{\omega(H_1\cap H_2)}}{\norm{H_1\cap H_2}}\\
   &\le \frac{s^2}{(\norm{\fp_j}-1)^2}\prod_{i<j}\left(1+\sum_{\nu\geq 1}\frac{2}{\norm{\fp_i}^\nu}\cdot (2\nu+1)\right)\\
   &\ll \frac{s^2}{\norm{\fp_j}^2}\prod_{i<j}\left(1+\frac{6}{\norm{\fp_i}}+O\left(\frac{1}{\norm{\fp_i}^2}\right)\right)\\
   &\leq \frac{s^2}{\norm{\fp_j}^2}\,\mathrm{exp}\left\lbrace \sum_{i<j}\frac{6}{\norm{\fp_i}}+O\left(\frac{1}{\norm{\fp_i}^2}\right)\right\rbrace\\
   &\ll \frac{s^2(\log\norm{\fp_j})^6}{\norm{\fp_j}^2}.
\end{align*}
Note that the last inequality in above calculations follows from Mertens' estimate for number fields, i.e., $\sum_{\norm{\fp}\leq x}1/\norm{\fp}=\log\log x+O(1)$, which follows from the prime number theorem for number fields (see \cite[p.~670]{Landau1903}) and partial summation. In conclusion, this completes the proof of part (2) of the lemma. 
\end{proof}

\section{Proof of Theorem~\ref{mainthm}}\label{sec_pf_mainthm}

In this section we finish the proof of Theorem~\ref{mainthm}. Let $\cA=\set{a_1+I_1, \dots, a_n+I_n}$ be a collection of congruence classes of multiplicity $s$ with $a_1,\dots, a_n\in \cO_K$ and ideals $I_1,\dots,I_n$ in $\cO_K$ with $\norm{I_1}\le \cdots\le \norm{I_n}$. To prove Theorem~\ref{mainthm}, it suffices to show that there exists a constant $c_K>0$ depending only on $K$ such that if
\[
\norm{I_1} > \exp\big(c_K\log^2(s+1)/\log\log(s+2)\big),
\]
then $\cA$ does not cover $\cO_K$. Moreover, using the notation in Section~\ref{sec_distortion_method}, by Theorem~\ref{thm_distortion_method_nf} it reduces to show that
\[
	\eta\colonequals \sum_{j=1}^J
	\min\left\{  M_j^{(1)},\frac{M_j^{(2)}}{4\delta_j(1-\delta_j)}\right\}<1. 
\]

Let $y = C s^3$, where $C$ is a large constant to be chosen later. Let $k = \max\{j \in [1, J] \cap \mathbb{Z}: \norm{\fp_j} \leq y\}$ and let 
\[
\delta_i = 
\begin{cases}
0 &\text{if } i \leq k, \\
\frac{1}{2} &\text{if } i > k. 
\end{cases}
\]
Then 
\[
\eta \leq \sum_{1 \leq j \leq k} M_j^{(1)} + \sum_{k < j \leq J} M_j^{(2)} =: \eta_1 + \eta_2. 
\]
Then it follows from Lemma~\ref{lem:moments} (2) that 
\begin{equation}\label{eqn_eta_2}
  \eta_2 \ll \sum_{\norm{\fp_j} > y} \frac{s^2 (\log \norm{\fp_j})^6}{\norm{\fp_j}^2} 
\ll \sum_{\norm{\fp_j} > y} \frac{s^2}{\norm{\fp_j}^{1.9}} 
\ll \frac{s^2}{y^{0.9}} 
= \frac{1}{C^{0.9} s^{0.7}}.   
\end{equation} 

Now, suppose $\norm{I_1} > x$, where $x$ is to be chosen later. To bound $\eta_1$,  we apply Lemma~\ref{lem:moments} (1)  to obtain that 

\begin{equation}\label{eqn_eta_1}
    \eta_1 
\le s \sum_{1 \leq j \leq k}
\sum_{\substack{\norm{I}\ge \norm{I_1} \\ P(I)=\norm{\fp_j}}} \frac{1}{\norm{I}} 
\le s 
\sum_{\substack{\norm{I}\ge x \\ P(I)\le y}} \frac{1}{\norm{I}},
\end{equation}
where $P(I)$ is the largest norm of prime factors of $I$. As regards the last sum above, we have the following estimate.

\begin{lemma}[{\cite[Lemma~2.6]{Kim2009}}]
\label{lem_smooth_number}
    Suppose $y\ge2$ and $y<x\le\exp(\exp(\log^{2/5}y))$. Let $u=\log x/\log y$, then
    \[
    \sum_{\substack{\norm{I}\ge x \\ P(I)\le y}} \frac{1}{\norm{I}} \ll (\log y)e^{-u\log u}.
    \]
\end{lemma}

To apply Lemma~\ref{lem_smooth_number}, we take 
\[
x \colonequals \exp \{\exp(\log^{1/6} C)\log^2(s + 1)/\log \log(s+2)\}.
\]
Then the assumptions in Lemma~\ref{lem_smooth_number} are satisfied for large enough $C$, and by \eqref{eqn_eta_1} we have
\begin{equation}\label{eqn_eta_1_2}
    \eta_1\ll s(\log y)e^{-u\log u}.
\end{equation}

By \eqref{eqn_eta_2} and \eqref{eqn_eta_1_2} above, we get that both $\eta_1$ and $\eta_2$ tend to  $0$ as $C\to\infty$. Therefore, we have $\eta < 1$ for sufficiently large $c_K$. This completes the proof.

\section{The \texorpdfstring{$j$}{j}-th smallest norm}\label{sectjthsmallestnorm}

Let $j\ge1$ be an integer. In \cite{KKL2022}, Klein, Koukoulopoulos, and Lemieux gave an upper bounded for  the $j$th smallest modulus in a minimal covering system with distinct moduli. In this section, as an application of Theorem~\ref{mainthm}, we obtain the following result.

\begin{theorem}\label{thmjthsmallestnorm}
    The $j$-th smallest norm in a minimal covering system of a number field with distinct moduli is bounded.
\end{theorem}

\begin{proof} We follow the argument in the proof of \cite[Theorem~1.5]{LWWY2024}. Let $K$ be a number field and $\cO_K$ be the ring of integers in $K$. We do induction on $j$. 

By Theorem~\ref{mainthm}, the case $j=1$ holds when we take $s=1$. Let $j>1$, and suppose that there are bounds $B_1,\dots,B_{j-1}$ such that the modulus with the $k$-th smallest norm in a minimal covering system of $\cO_K$ with distinct moduli is bounded by $B_k$  for each $1 \leq k < j$. 
	
Now, let $\cC=\set{a_i+ I_i \colon 1 \leq i \leq \ell}$ be a minimal covering system of $\cO_K$  with distinct moduli and $\ell \geq j$ congruences, and $\norm{I_1}\leq \cdots \leq \norm{I_\ell}$. By the assumption, we have that $\norm{I_k}\leq B_k$ for $1\leq k <j$. Let $L=\on{lcm}[I_1, \dots, I_{j-1}]$. Then $\norm{L}\leq B\colonequals\prod_{k=1}^{j-1}B_k$. We write $\cC=\cC_1 \cup \cC_2$, where $\cC_1=\set{a_k+ I_k \colon 1 \leq k <j}$ and $\cC_2=\set{a_i + I_i \colon j \leq i \leq \ell}$. Since $\cC$ is minimal, the congruences in $\cC_1$ do not form a covering system, hence the congruences in $\cC_2$
cover at least one congruence class modulo $L$, say $h+ L$. Let
\[
\cC_3\colonequals\set{a_i+r+I_i\colon j \leq i \leq \ell, r \in \cO_K/L },
\]
then $\cC_3$ is a covering system. Indeed, for any $a \in \cO_K$
, there exists some $b \in \cO_K/L $ such that $a -b \equiv h \mod L$. It follows that $a-b+L$ is covered by $\cC_2$. Thus, there exists some $j\leq i\leq \ell$ such that $a -b \equiv a_i \mod{I_i}$.  It follows that $a\in a_i+b+I_i$. Hence $\cC_3$ is a covering system. Notice that the multiplicity of $\cC_3$ is at most $s=\norm{L} \leq B$, and $I_j$ is the ideal with smallest norm in $\cC_3$. By Theorem~\ref{mainthm}, we have $$\norm{I_j}\le \exp\big(c_K\log^2(B+1)/\log\log(B+2)\big).$$ 
In particular, $\norm{I_j}$ is bounded. Thus, the desired result follows by induction.
\end{proof}

\textbf{Acknowledgements.} The authors would like to thank Chunlin Wang for helpful discussions. 

\textbf{Funding.} 
This work is supported by the National Natural Science Foundation of China (Grant No. 12561001). Huixi Li's research was also supported by the National Natural Science Foundation of China (Grant No. 12201313). Shaoyun Yi is supported by the National Natural Science Foundation of China (Nos.~12301016, 12471187) and the Fundamental Research Funds for the Central Universities (No.~20720230025).

\end{document}